\documentclass{amsart}
\usepackage[utf8]{inputenc}
\usepackage[T1]{fontenc}

\usepackage{hyperref}
\usepackage{amsfonts}
\usepackage{eurosym}
\usepackage{amssymb}
\usepackage{dsfont}
\usepackage{color}
\usepackage{graphicx}
\usepackage{tikz}
\usetikzlibrary{cd}
\usepackage[colorinlistoftodos]{todonotes}
\usepackage[normalem]{ulem}

\newcommand{\N}{\mathbb N}

\newcommand{\prob}{\mathcal Prob}
\newtheorem{thm}{Theorem}

\newtheorem{lem}[thm]{Lemma}


\title{On $n$-saturated closed graphs}
\date{}
\author[S. G\l \c{a}b]{Szymon G\l \c{a}b}
\address{Institute of Mathematics, \L \'od\'z University of Technology, ul.
W\'olcza\'nska 215, 93-005 \L \'od\'z, Poland}
\email{szymon.glab@p.lodz.pl}

\author[P. Gordinowicz]{Przemysław Gordinowicz}
\address{Institute of Mathematics, \L \'od\'z University of Technology, ul.
W\'olcza\'nska 215, 93-005 \L \'od\'z, Poland}
\email{pgordin@p.lodz.pl}

\begin{document}

\maketitle

\begin{abstract}
Geschke proved in \cite{Geschke} that there is clopen graph on $2^\omega$ which is 3-saturated, but the clopen graphs on $2^\omega$ do not even have infinite subgraphs that are 4-saturated; however there is $F_\sigma$ graph that is $\omega_1$-saturated. It turns out that there is no closed graph on $2^\omega$ which is $\omega$-saturated, see \cite{GGK}. In this note we complete this picture by proving that for every $n\in\N$ there is an $n$-saturated closed graph on the Cantor space $2^\omega$. The key lemma is based on probabilistic argument. The final construction is an inverse limit of finite graphs.  
\end{abstract}

\maketitle

A graph $G$ is $n$-saturated if for any set $A \subseteq V(G)$ of its vertices, $\vert A \vert = n-1$ and any subset $B \subseteq A$  there is a vertex $w\in V(G) \setminus A$ which is adjacent to all vertices from $A \setminus B$ and to no vertex of $B$. The random graph \cite{Rado} is known as a countable graph which is $\omega$-saturated, that is $n$-saturated for every $n$. Saturation was also studied for topological graphs. A graph $G$ on a topological space $X$ is clopen (open) if the edge relation of $G$ is a clopen (open) subset of $X^2$ without the diagonal. A graph $G$ on $X$ is called compact (closed, $F_\sigma$, etc.) if the edge relation of $G$ is a compact (closed, $F_\sigma$, etc.) subset of $X^2$. Geschke proved in \cite{Geschke} that there is clopen graph on $2^\omega$ which is 3-saturated, but the clopen graphs on $2^\omega$ do not even have infinite subgraphs that are 4-saturated; however there is $F_\sigma$ graph that is $\omega_1$-saturated. It turns out that there is no closed graph on $2^\omega$ which is $\omega$-saturated, see \cite{GGK}. So there is a natural question whether there exist $n$-saturated closed graphs on $2^\omega$, for finite $n>3$? We answer this question in positive. This makes our knowledge of topological graph saturation more complete. Our construction uses different means comparing to that in \cite{Geschke}. It utilizes a probabilistic argument in the key lemma and an inverse limit of finite graphs in the final construction.  

\section{Introduction}

By a graph we understand a pair $G=(V(G),E(G))$, where $V(G)$ is a non-empty set of vertices and $E(G)$ is a symmetric and reflexive relation on $V(G)$. The reflexivity of $E(G)$ means that each vertex of the graph $G$ has a loop. A homomorphism of graphs is a map $h \colon V(G)\to V(H)$ that preserves edges. A graph homomorphism $h \colon V(G)\to V(H)$ is called strict if for every edge $(p,q) \in E(H)$ such that $p,q \in h(V(G))$ there is an edge $(a,b) \in E(G)$ such that $h(a)=p$ and $h(b)=q$. A surjective strict homomorphism is called a \emph{quotient map}. Note that having loops in considered graphs is needed to being able to map an edge onto a single vertex without violating edge preserving property. That is basically the only reason to consider graphs with loops. 

Let $G$ be a graph, $A\subseteq V(G)$. A type over $A$ is a function $f\in\{0,1\}^A$ (that means $f\colon A\to\{0,1\}$). A vertex $v\in V(G)\setminus A$ realizes type $f$ over $A$ provided that for every $a\in A$, $a$ and $v$ are adjacent if and only if $f(a)=1$. We say that $G$ is an \emph{$n$-saturated graph} if for every $A\subseteq V(G)$, $\vert A\vert<n$ and any type $f\in\{0,1\}^A$ there is $x\in V$ that realizes type $f$; in other words $(a,x)\in E(G)\iff f(a)=1$ for every $a\in A$. Further, we say that $G$ is a \emph{weakly $n$-saturated graph} if for every $A\subseteq V(G)$, $\vert A\vert<n$, there is $x\in V(G)$ which is adjacent to every $a\in A$. 

Our final graph will be constructed as an inverse limit of finite graphs. Let us briefly recall the definition of the inverse limit in the graph context. 
Assume that $\{ G_n \colon n\in\N\}$ is a family of finite graphs and $p_{n} \colon V(G_{n+1})\to V(G_n)$ are quotient maps for $n\in\N$.  Let $p^n_{k}=p_{k}\circ p_{k+1}\circ\cdots\circ p_{n-1}$ for $n>k\geq 1$. Then $p^n_k\colon V(G_{n})\to V(G_k)$ is a quotient map as a superposition of quotient maps. Let $G:=\varprojlim G_n$ be the inverse limit of $\{(G_n)_{n\in\N},(p^n_{k})_{n>k}\}$, that is the graph with the set of vertices
$$
V(G)=\Big\{a\in \prod_{n\in\N}V(G_n) \colon a(k)=p^n_{k}(a(n))\text{ for }n>k\Big\}
$$
and the edge relation $E(G)$ given by
$$
(a,b)\in E(G)\iff\forall n\in\N\;\;(a(n),b(n))\in E(G_n).
$$
Clearly $E(G)$ symmetric and reflexive, and therefore $G$ is a graph. Assume that each finite graph $G_n$ has discrete topology. On the product $\prod V(G_n)$ we consider the product topology, that is the topology given by basic sets of the form
$$
B(x_1,\dots,x_n):=\Big\{a\in\prod_{n\in\N}V(G_n) \colon a(i)=x_i\text{ for }i\leq n\Big\}
$$
where $x_i\in V(G_i)$ for $i\leq n$. The space $\prod V(G_n)$ is metrizable, compact, zero-dimensional and perfect (i.e. it has no isolated points). Therefore by the Brouwer Theorem \cite[7.4]{K} it is homeomorphic to the Cantor space $\{0,1\}^\N$. We consider $G$ with the topology inherited from $\prod V(G_n)$. It turns out that $V(G)$ is a closed subset of $\prod V(G_n)$ and $E(G)$ is closed subset of $V(G)\times V(G)$. Moreover $q_n \colon V(G)\to V(G_n)$ given by $q_n(a)=a(n)$ is a quotient map for every $n\in\N$. Graphs of the form $\varprojlim G_n$, where $G_n$ are finite, are called profinite graphs. It turns out that profinite graphs coincide with compact graphs on zero-dimensional metrizable compact spaces, see \cite{GGK}. 

Closed subset of a compact metrizable zero-dimensional space is again compact metrizable and zero-dimensional. Note that $a$ is an isolated point in the compact metrizable zero-dimensional topological space $V(G)$ if there is $n$ such that $\{a\}\cap B(x_1,\dots,x_n)\cap V(G)=\{a\}$ where $x_i=a(i)$. We say that \textit{every vertex in $(G_n)$ eventually splits}, that is for any $n\in\N$ and $v\in V(G_n)$ there are $m>n$ and two distinct $x,y\in V(G_m)$ with $p^m_n(x)=p^m_n(y)=v$. This condition ensures us that if a basic set $B(x_1,\dots,x_n)\cap V(G)$ is non-empty, it is not a singleton. Finally, if every vertex in $(G_n)$ eventually splits, then $\varprojlim G_n$ is perfect, and consequently it is homeomorphic to the Cantor space.

\section{Construction of $n$-saturated closed graph on the Cantor space}

Consider the following random construction. For a fixed $n \in \N$ we start with a weakly $n$-saturated finite graph $H_0$ having vertex set $V(H_0)=\{1,2,\dots,k\}$ for some $k \geq n$. By $H_m$ we denote a random graph with vertex set $V(H_m) = \{1,2,\dots,k\}\times\{0,1,\dots,m\}$ and the edge relation defined as follows 
\begin{itemize}
\item[(i)] For every $i<j\leq k$ $(i,0)$ and $(j,0)$ are adjacent in $H_m$ $\iff i$ and $j$ are adjacent in $H_0$;
\item[(ii)] For any $i,j\leq k$ if $i$ and $j$ are not adjacent in $H_0$, then $(i,s)$ and $(j,t)$ are not adjacent as well for any $s,t\leq m$;
\item[(iii)] If $i$ and $j$ are adjacent in $H_0$ and $s,t\leq m$ with $s^2+t^2>0$ (at least one of $s$ and $t$ is greater than zero), then $(i,s)$ and $(j,t)$ are adjacent in $H_m$ with probability $1/2$, and the decision -- whether there {is such an edge or not} -- is made independently to the others.
\item[(iv)] $H_m$ is reflexive.
\end{itemize}

\begin{lem}\label{ProbGraph}
There is $m\in\N$ such that
\begin{itemize}
\item[(A)] $\prob(H_m$ is $n$-saturated$)>0$;
\item[(B)] $\prob(\mathcal{X})>0$ where $\mathcal{X}$ is an event that for every $p<n$, $i_1,\dots,i_p,i\in V(H_0)$ such that $i_1,\dots,i_p$ are adjacent to $i$ in $H_0$, and every $j_1,\dots,j_p\in\{0,1,\dots,m\}$ there is $l\in\{0,1,\dots,m\}$ such that $(i_1,j_1),\dots,(i_p,j_p)$ are adjacent to $(i,l)$ in $H_m$.   
\end{itemize}
\end{lem}

\begin{proof}
Let $A \subseteq V(H_m)$, $\vert A\vert=n-1$, say $A=\{(i_1,j_1),\dots,(i_{n-1},j_{n-1})\}$. By (i) the subgraph $\{(i,0) \colon i=1,2,\dots,k\}$ of $H_m$ is isomorphic to $H_0$, which in turn is weakly $n$-saturated. Therefore there is $i\in V(H_0)\setminus\{i_1,\dots,i_{n-1}\}$ adjacent to each $i_1,\dots,i_{n-1}$ in $H_0$. Fix a type $f\in\{0,1\}^A$. The probability that a given vertex $(i,l)$ in $H_m$ realizes the type $f$ equals $1/2^{n-1}$. Thus the probability that none vertex $(i,l)$ in $H_m$ realizes $f$ equals 
$$
\prob\Big(\bigcap_{l\leq m} \{(i,l)\text{ does not realize type }f\}\Big)=\Big(1-\frac{1}{2^{n-1}}\Big)^m. 
$$ 
Therefore we obtain
$$
\prob(H_m\text{ is not }n\text{-saturated})=
\prob\Big(\bigcup_{\vert A\vert=n-1}\bigcup_{f\in \{0,1\}^A} \{f\text{ is not realized in }H_m\}\Big)
$$
$$
\leq \sum_{\vert A\vert=n-1}\sum_{f\in \{0,1\}^A}\Big(1-\frac{1}{2^{n-1}}\Big)^m={(m+1)k \choose n-1}2^{n-1}\Big(1-\frac{1}{2^{n-1}}\Big)^m.
$$
Note that the latter number tends to zero when $m\to\infty$. Hence, for large enough $m$ there is $\prob(H_m\text{ is }n\text{-saturated})>0$. 

Note that the proof of part (B) of the lemma is analogous to the above construction. The difference is that a vertex $i \in V(H_0)$ is already chosen and the type function $f$ is constantly equal to 1. 
\end{proof}

\begin{lem}\label{QuotientMap} Let $H_m$ be a graph which satisfies the assertion of Lemma \ref{ProbGraph}.
Let $p \colon V(H_m) \to V(H_0)$ be given by $p(i,j)=i$. Then $p$ is a projection. Moreover, whenever, for $k<n$, vertices $v_1,\dots,v_k$ are adjacent to $v$ in $V(H_0)$, then for every $w_1,\dots,w_k \in V(H_m)$ with $p(w_i)=v_i$ there is an $w\in H_m$ which is adjacent to every $w_1,\dots,w_k$ and $p(w)=v$.
\end{lem}

\begin{proof}
Assume that $((i,s),(j,t))\in E(H_m)$. If $i=j$, then $(i,j)\in E(H_0)$, since $E(H_0)$ is reflexive. If $i\neq j$ and $s=t=0$, then by (i) we obtain $(i,j)\in E(H_0)$. If $i\neq j$ and $s^2+t^2>0$, then $(i,j)\in E(H_0)$ by (iii) and (ii). Therefore $p$ is a homomorphism. By (i) it is also strict and surjective, and therefore it is a quotient map. Moreover part of assertion follows from condition (B) of Lemma \ref{ProbGraph}.
\end{proof}
  
\begin{thm}
Let $n\in\N$. There exists $n$-saturated closed graphs on $2^\omega$. 
\end{thm}

\begin{proof}
We will construct a profinite graph which is $n$-saturated as an inverse limit of finite graphs. Let $G_0$ be a complete graphs with $n$ vertices. Clearly $G_0$ is weakly $n$-saturated. Using Lemma \ref{ProbGraph} and Lemma \ref{QuotientMap} there is a finite $n$-saturated graph $G_1$ and a quotient map $p_{0} \colon V(G_1)\to V(G_0)$. Proceeding inductively we find sequences  $(G_n)_{n\in\N}$ and $(p_{n})_{n\in\N}$ of finite $n$-saturated graphs and of quotient mappings, respectively, such that  $p_{n}:V(G_{n+1})\to V(G_{n})$. Let $G$ be its inverse limit. In the construction of $H_m$, $p(i,0)=p(i,1)=i$ for any $i\in V$. This means that every vertex in $(G_n)$ eventually splits, and therefore $G$ is, as a topological space, homeomorphic to the Cantor space $2^\omega$. We will show that $G$ is $n$-saturated.

Let $A\subseteq V(G)$, $\vert A\vert=n-1$ and $f\in\{0,1\}^A$. Then $A=\{x_1,\dots,x_{n-1}\}$, $x_i(m)\in V(G_{m})$, $p_{m}(x_i(m+1))=x_i(m)$, for every $m\in\N$. Setting $k = \sum_{i} f(x_i)$ we may assume that $f(x_i)=1$ if and only if $i\leq k$.  There is $m$ such that the set $p_m(A)=\{x_1(m),\dots,x_{n-1}(m)\}\subseteq V(G_m)$ has $n-1$ elements. There is $v_m\in G_m$ which realizes $f$, that means $v_m$ is adjacent to vertices $x_1(m),\dots, x_k(m)$ and not adjacent to $x_{k+1}(m),\dots,x_{n-1}(m)$. By Lemma \ref{QuotientMap} there is $v_{m+1}\in G_{m+1}$ which is adjacent to $x_1(m+1),\dots, x_k(m+1)$ and $p_{m}(v_{m+1})=v_{m}$. Proceeding inductively {for any $i>m$} we find $v_i$ such that $v_{i}\in V(G_{i})$ which is adjacent to $x_1(i),\dots, x_k(i)$ and $p_{i-1}(v_{i})=v_{i-1}$.  

Define $v\in V(G)$ as follows 
$$
v(i) = \left\{
\begin{array}{ccl}
v_i & \text{~for~} & i\geq m,\\
p^m_{i}(v_m) & \text{~for~} & i < m.
\end{array}
\right.
$$
Then $v(i)$ is adjacent to $x_1(i),\dots,x_k(i)$ for every $i$ (for $i<m$ it follows from the fact that projection $p^m_i$ preserves edges). Therefore $v$ is adjacent to $x_1,\dots,x_k$. Note that $v$ is not adjacent to $x_{k+1},\dots,x_n$, since $v(m)$ is not adjacent $x_{k+1}(m),\dots,x_{n-1}(m)$.
\end{proof}

\end{document}